\documentclass[10pt,reqno]{amsart}

\usepackage{amssymb,latexsym,mathrsfs}
\usepackage{url}
\usepackage{enumerate}[1.)]
\numberwithin{equation}{section}

\renewcommand{\eqref}[1]{(\ref{#1})}   %for some reason eqref is not supported properly
\theoremstyle{plain}
\newtheorem{theorem}{Theorem}[section]

\newtheorem{corollary}[theorem]{Corollary}

\newtheorem{lemma}[theorem]{Lemma}

\newtheorem{proposition}[theorem]{Proposition}

\theoremstyle{definition}
\newtheorem{definition}[theorem]{Definition}
\newtheorem{example}[theorem]{Example}

\theoremstyle{remark}
\newtheorem{remark}[theorem]{Remark}

\renewcommand{\leq}{\leqslant}
\renewcommand{\geq}{\geqslant}

\newcommand{\Fp}{\mathbb{F}_p}
\newcommand{\Fpt}{\mathbb{F}^\times_p}
\newcommand{\Cc}{\mathbb{C}}

\newcommand{\Zz}{\mathbb{Z}}
\newcommand{\Ff}{\mathbb{F}}
\newcommand{\Aa}{\mathbb{A}}

\newcommand{\Qq}{\mathbb{Q}}

\DeclareMathOperator{\cond}{cond}
\newcommand{\sheaf}[1]{\mathcal{{#1}}}
\newcommand{\frtr}[3]{(\Tr{{#1}})({#2},{#3})}
\newcommand{\corr}[2]{\mathcal{C}({#1},{#2})}
\DeclareMathOperator{\Tr}{tr}
\DeclareMathOperator{\swan}{Sw}
\DeclareMathOperator{\rank}{rank}

\newcommand{\lra}{\longrightarrow}
\newcommand{\ra}{\rightarrow}
\newcommand{\eps}{\varepsilon}

\begin{document}

\title{The sliding-sum method for short exponential sums}

\date{\today}

\author[\'E. Fouvry]{\'Etienne Fouvry}
\address{Universit\'e Paris Sud, Laboratoire de Math\'ematique\\
  Campus d'Orsay\\ 91405 Orsay Cedex\\France}
\email{etienne.fouvry@math.u-psud.fr}

\author[E. Kowalski]{Emmanuel Kowalski}
\address{ETH Z\"urich -- D-MATH\\
  R\"amistrasse 101\\
  CH-8092 Z\"urich\\
  Switzerland} \email{kowalski@math.ethz.ch}

\author[Ph. Michel]{Philippe Michel}
\address{EPFL/SB/IMB/TAN, Station 8, CH-1015 Lausanne, Switzerland }
\email{philippe.michel@epfl.ch}

\thanks{Ph. M. was partially supported by the SNF (grant
  200021-137488) and the ERC (Advanced Research Grant 228304);
  \'E. F. thanks ETH Z\"urich, EPF Lausanne and the Institut
  Universitaire de France for financial support. }



\begin{abstract}
  We introduce a method to estimate sums of oscillating functions on
  finite abelian groups over intervals or (generalized) arithmetic
  progressions, when the size of the interval is such that the
  completing techniques of Fourier analysis are barely insufficient to
  obtain non-trivial results. In particular, we prove various
  estimates for exponential sums over intervals in finite fields and
  related sums just below the Polya-Vinogradov range, and derive
  applications to equidistribution problems.
\end{abstract}

\maketitle
%%\tableofcontents

\section{Introduction} 

\noindent\rule{\textwidth}{1pt}
\textbf{Remark.} Theorem~\ref{th-slide} has been
significantly strenghtened in the joint work~\cite{shortsums} with
CS. Raju, J. Rivat and K. Soundararajan; this paper also contains
further results on short sums of trace functions. However, the results
of Section~\ref{sec-gap} concerning sums over generalized arithmetic
progressions are not contained in~\cite{shortsums}, and the present
paper is left on arXiv for this reason.
\par
\noindent\rule{\textwidth}{1pt}
\par
\medskip
\par
A basic idea in analytic number theory, with countless applications,
is the \emph{completion technique}, which gives estimates for sums
over short intervals of integers in terms of longer sums by means of
Fourier techniques. If we denote by $\hat{\varphi}$ the discrete
Fourier transform of a complex-valued function $\varphi$ defined on
$\Zz/m\Zz$, $m\geq 1$, normalized by defining
$$
\hat{\varphi}(t)=\frac{1}{\sqrt{m}}\sum_{n\in \Zz/m\Zz}\varphi
(n)e\Bigl(\frac{nt}{m}\Bigr),\quad\text{ where }\quad e(z)=e^{2i\pi
  z},
$$
then the basic inequality for sums of values of $\varphi$ over an
interval $I$ of length $<m$ projected modulo $m$ is
\begin{align}
  \Bigr\vert \sum_{n\in I} \varphi(n) \Bigl\vert &= \Bigl|
  \sum_{t\in\Zz/m\Zz} \hat{\varphi}(t) \hat{I}(t)\Bigr|
  \nonumber  \\
  &\leq \|\hat{\varphi}\|_{\infty} m^{1/2}(\log 3m)\label{eq-fourier}
\end{align}
where $\hat{I}$ the Fourier transform of the characteristic function
of $I$, which satisfies
$$
\sum_{t\in\Zz/m\Zz}|\hat{I}(t)|\leq \sqrt{m}\ (\log 3m).
$$
\par
A very classical application arises when $m=p$ is prime and
$$
\varphi(n)=\chi(f(n))e\Bigl(\frac{g(n)}{p}\Bigr)
$$
where $\chi$ is a Dirichlet character modulo $p$ and $f, g\in \Qq(X)$
are fixed rational functions, since one can then build on Weil's
theory of exponential sums in one variable over finite fields to
estimate the $L^{\infty}$-norm of the Fourier transform of
$\varphi$. To give a precise statement, assume $f=f_1/f_2$,
$g=g_1/g_2$ with $f_i\in \Zz[X]$, $g_i\in \Zz[X]$ monic
polynomials. Then we have
\begin{equation}\label{eq-polya-vino-chars}
  \sum_{n\in I} \chi(f(n))e\Bigl(\frac{g(n)}{p}\Bigr)\ll
  (\deg(f_1)+\deg(f_2)+\deg(g_1)+\deg(g_2))\sqrt{p}(\log p), 
\end{equation}
where the implied constant is absolute, for all primes $p$ such that
at least one of the following conditions holds:
\par
-- the character $\chi$ is of order $h\geq 1$ and $f$ modulo $p$ is
not proportional to an $h$-th power in $\Fp(X)$;
\par
-- the rational function $g$ modulo $p$ is not proportional to a
polynomial of degree at most $1$.
\par
In the special case where $\chi$ is non-trivial, $f=X$ and $g=0$, this
result is the classical \emph{Polya-Vinogradov inequality} (for prime
moduli). In all cases, it is clear that such an estimate is
non-trivial as long as $I$ is of length \emph{at least} $\gg
\sqrt{p}(\log p)$.
\par
In this generality, the result is \emph{almost} best possible, since
for
$$
\varphi(n)=e\Bigl(\frac{n^2}{p}\Bigr),
$$
the sum over $1\leq n\leq p^{1/2}$ exhibits no significant
cancellation. Although the gap between $p^{1/2}$ and $p^{1/2}(\log p)$
is small, it is natural to ask whether it should exist or not. We will
show in this note that, for many natural functions $\varphi$,
including those above, one gets some cancellation as long as
$p^{1/2}=o(|I|)$. The functions we use are, as in our previous works,
the \emph{trace functions} modulo primes (see Section~\ref{alggeo} for
reminders and examples; these functions satisfy a general form
of~(\ref{eq-polya-vino-chars}), see Remark~\ref{rm-general}.) A
special case is the following:

\begin{theorem}[Sliding sum bound]\label{th-slide}
  Let $p$ be a prime number, $c\geq 1$, and let $\varphi$ be an
  isotypic Fourier trace function modulo $p$, of conductor
  $\cond(\varphi)\leq c$ \emph{(}for instance
$$
\varphi(n)=\chi(f(n))e\Bigl(\frac{g(n)}{p}\Bigr)
$$
where $f$, $g\in\Qq(X)$ satisfy one of the two conditions above with 
$c\leq \deg(f_1)+\deg(f_2)+\deg(g_1)+\deg(g_2)$\emph{)}. 
\par
Then, for any interval $I$ in $\Zz/p\Zz$ with $|I|>\sqrt{p}$, we have
$$
\sum_{n\in I}\varphi(n)\ll
c^{4}|I|\Bigl(\frac{p^{1/2}}{|I|}\Bigr)^{1/3},
$$
where the implied constant is absolute.
\end{theorem}

We will derive this from a rather simple general inequality which
offers wider possibilities for applications (e.g., we apply it in
Section~\ref{sec-mult} to sums of trace functions over geometric
progressions in a finite field.) We then apply these bounds to derive
equidistribution results which, again, bridge the gap between
$\sqrt{p}$ and $\sqrt{p}(\log p)$.

\begin{corollary}[Equidistribution]
\label{cor-rational}
Let
$\beta$ be any function defined on positive integers such that $1\leq
\beta(p)\ra +\infty$ as $p\ra +\infty$, and for all $p$ prime, let
$I_p$ be an interval in $\Fp$ of length $|I_p|\geq p^{1/2}\beta(p)$. 
\par
\emph{(1)} Let $f_1$, $f_2\in\Zz[X]$ be monic polynomials such that
$f=f_1/f_2\in\Qq(X)$ is not a polynomial of degree $\leq 1$. Then for
$p$ prime, the set of fractional parts
$$
\Bigl\{\frac{f(n)}{p}\Bigr\},\quad\quad n\in I_p,
$$
becomes equidistributed in $[0,1]$ with respect to Lebesgue measure as
$p\ra +\infty$, where $f(n)$ is computed in $\Fp$.
\par
\emph{(2)} For $p$ prime and $n\in\Fpt$, define the Kloosterman angle
$\theta_p(n)\in [0,\pi]$ by the relation
$$
\frac{S(n,1;p)}{\sqrt{p}}=\frac{1}{\sqrt{p}}
\sum_{x\in\Fpt}e\Bigl(\frac{nx+\bar{x}}{p}\Bigr)=2\cos\theta_p(n).
$$
\par
Then the angles $\theta_p(n)$ for $n\in I_p$ become equidistributed in
$[0,\pi]$ with respect to the Sato-Tate measure $2\pi^{-1}\sin^2
\theta d\theta$.
\end{corollary}

For $\beta(p)/(\log p)\ra +\infty$, this follows from the
Polya-Vinogradov bound and Weil's method (for the first part) or the
extension by Michel~\cite[Cor. 2.9, 2.10]{michel} of the
equidistribution results of Katz~\cite{katz-gskm} for angles of
Kloosterman sums (for the second), but as far as we know, this
extension was not previously known in general. The first result cannot
be improved in general since $f(n)=n^2$ is a counterexample when
$\beta$ is a bounded function. We will also give another similar
application to the distribution of ``polynomial residues'' (see
Proposition~\ref{pr-pol-residues}).

\begin{remark}
  We recall that in many cases, one does expect non-trivial estimates
  for much shorter sums, but only relatively little progress has been
  made concerning this problem. We only recall two of the most
  classical results: when $\varphi(m) =\chi (m)$ is a non-trivial
  multiplicative character modulo $p$, the Burgess bound (see,
  e.g.,~\cite[Theorem 12.6]{IK}) is non-trivial for intervals of
  length $\gg p^{\frac{1}{4} +\eps}$ for any $\eps>0$; for $\varphi(m)
  =e(m^k/p)$, where $k\geq 3$ is an an integer, Weyl's method gives
  non-trivial estimates for intervals of length $\gg p^{\frac{1}{k}
    +\varepsilon}$ for any $\varepsilon >0$ (see, e.g.~\cite[\S
  8.2]{IK}.)
\end{remark}

\section{The sliding sum method}

Our basic inequality is very simple, and vaguely reminiscent of van
der Corput's shift inequality (see, e.g.,~\cite[Lemma
8.17]{IK}). % Beyond the case of short intervals modulo $m\geq 1$, it
% applies to sums over any subset of a finite abelian group which is
% ``almost invariant'' under sufficiently many shifts.
\par
We will use the following notation. Let $A$ be a finite abelian
group. For any subset $B\subset A$ and any function
$$
\varphi\,:\, A\rightarrow \Cc
$$
we denote
\begin{equation}\label{defS}
  S(\varphi;B)=\sum_{x\in B}\varphi(x).
\end{equation}
\par
We also define 
$$
\|\varphi\|_{\infty}=\max_{x\in A} |\varphi(x)|,\quad\quad
\|\varphi\|_{2}=\Bigr(\sum_{x\in A}|\varphi(x)|^2 \Bigl)^\frac{1}{2},
$$
and the \emph{additive correlations} of $\varphi$ given by
$$
\corr{\varphi}{a}=\sum_{x\in A}\varphi(x)\overline{\varphi(x+a)}
$$
for $a\in A$. We note that, by the Cauchy-Schwarz inequality, we have
\begin{equation}\label{eq-bound-correlations}
  |\corr{\varphi}{a}|\leq \|\varphi\|_2^2
\end{equation}
for all $a\in A$.
\par
% Since the case of intervals is likely to be the most important for
% applications, we will begin by describing our basic inequality in
% that case.
When $A=\Zz/m\Zz$ is a finite cyclic group, we define an
\emph{interval in $A$} to be a subset $B$ which is the reduction
modulo $m$ of an interval of consecutive integers, such that the
reduction is injective.

\begin{theorem}[Sliding-sum bound]\label{central} 
With notation as above, for any $m\geq 1$, any function $\varphi$ on
$A=\Zz/m\Zz$, any interval $I$ in $A$, and any subset $D\subset A$, we
have
\begin{multline}\label{eq-sliding}
  |S(\varphi;I)|\leq 
  2\|\varphi\|_{\infty}^{1/3}
  \Bigl\{
  |D|^{1/3}|I|^{1/3}\|\varphi\|_2^{2/3}+
  |I|^{2/3}\max_{a\notin D}|\corr{\varphi}{a}|^{1/3}\\
+\frac{2}{3}|I|^{2/3}\|\varphi\|_{\infty}^{2/3}
  \Bigr\}.
\end{multline}
\end{theorem}

We write this bound in terms of exact constants but it might be easier
to understand asymptotically as $m\ra +\infty$, thinking of the size
of $D$ and of the $L^{\infty}$-norm of $\varphi$ as quantities which
remain bounded by absolute constants while $m\ra +\infty$, and viewing
the last term as of smaller order of magnitude than the second (which
is almost universally true). Estimating the $L^2$-norm in terms of the
$L^{\infty}$-norm, the bound becomes roughly of order of magnitude
$$
p^{1/3}|I|^{1/3}+|I|^{2/3}\max_{a\notin D}|\corr{\varphi}{a}|^{1/3},
$$
where one can see already that the first term is $o(|I|)$ provided
$p^{1/2}=o(|I|)$; see Sections~\ref{sec-abstract} and~\ref{alggeo} for
discussion of the estimate of the second term, and for instances where
the asymptotic assumptions we described are reasonable.
\par

We now prove Theorem~\ref{central}, but first we isolate the property
of the interval $I$ that is used: given $a\in A=\Zz/m\Zz$ and an
interval $I$, we have
\begin{equation}\label{eq-almost-invariant}
|T_s(I)|\geq s
\end{equation}
for all $s\leq m$, where
\begin{equation}\label{eq-tsb}
T_s(I)=\{a\in A\,\mid\, |(a+I)\triangle I|\leq s\},
\end{equation}
with $\triangle$ denoting the symmetric difference. This is a
statement of ``almost'' invariance under additive shifts.  
\par 
Indeed, for any integer $a\in\Zz$, we have
$$
|I\triangle (a+I)|\leq 2|a|.
$$
\par
For an integer $s$ with $1\leq s<m$, the integers $a$ with $2|a|\leq
s$ are distinct modulo $m$, and thus we get
$$
|T_s(I)|\geq \sum_{2|a|\leq s}1=2\Bigl\lfloor
\frac{s}{2}\Bigr\rfloor+1 \geq s.
$$
\par
This remark concerning intervals has some interest, because the
property involved applies to at least another example.

\begin{example}
  Let $H\subset A$ be a \emph{subgroup} of $A$. Then for $a\in A$, we
  have
$$
|H\triangle (a+H)|=\begin{cases}
0&\text{ if } a\in H\\
2|H|&\text{ if } a\notin H,
\end{cases}
$$
since both $H$ and $a+H$ are $H$-cosets. Thus we have
$$
|T_s(H)|=\begin{cases}
|H|&\text{ if } s<2|H|\\
|A|&\text{ if } s\geq 2|H|,
\end{cases}
$$
and hence, for $1\leq s<2|H|$, we have
$$
|T_s(H)|\geq |H|\geq \frac{s}{2},
$$
which is very close to~(\ref{eq-almost-invariant}). Thus, the proof
below shows that the estimate of Theorem~\ref{central} applies, up to
a multiplicative factor, when $I$ is replaced by an arbitrary subgroup
of a finite abelian group $A$. This may be useful when $A$ is far from
cyclic, e.g., for the multiplicative group $\Fpt$ when $p-1$ has many
prime factors.
\end{example}

\begin{proof}[Proof of Theorem~\ref{central}]
  We can assume $\varphi\not=0$. We will then compare upper and lower
  bounds for the average
\begin{equation}\label{defSigma}
  \Sigma=\sum_{a\in A}\Bigl|\sum_{x\in B}\varphi(x+a)\Bigr|^2=
  \sum_{a\in A}|S(\varphi;B+a)|^2.
\end{equation}
\par
For the upper-bound, we expand the square and exchange the order of
summation, obtaining
\begin{align*}
  \Sigma&=\sum_{x,y\in B}\sum_{a\in
    A}\varphi(x+a)\overline{\varphi(y+a)} \\& =\sum_{x,y\in
    B}\sum_{a\in A}\varphi(a)\overline{\varphi(a+y-x)} =\sum_{x,y\in
    B}\corr{\varphi}{y-x}.
\end{align*}
\par
We split the sum according to whether $y-x$ is in $D$ or not. The
contribution of the $x$ and $y$ such that $y-x\notin D$ satisfies
$$
\Bigl|\sum_{\substack{x,y\in B\\y-x\notin D}}\corr{\varphi}{y-x}
\Bigr|\leq |B|^2\max_{a\notin D}|\corr{\varphi}{a}|,
$$
while, using~(\ref{eq-bound-correlations}), we have
$$
\Bigl|\sum_{\substack{x,y\in B\\y-x\in D}}\corr{\varphi}{y-x}
\Bigr|\leq \|\varphi\|_2^2\sum_{\substack{x,y\in I\\y-x\in D}}1 \leq
\|\varphi\|_2^2|B||D|.
$$
\par
Hence we get the upper-bound
$$
\Sigma\leq |B|\ |D|\ \|\varphi\|_2^2+|B|^2\max_{a\notin D}|
\corr{\varphi}{a}|.
$$
\par
For the lower-bound, let
$$
s=\Bigl\lfloor\frac{1}{2}\frac{|S(\varphi;B)|}{\|\varphi\|_{\infty}}
\Bigr\rfloor\leq \frac{|B|}{2}
$$
and use positivity to restrict the sum to $a\in T_s(B)$. For any $a\in
T_s(B)$, the set defined in~(\ref{eq-tsb}), we have
$$
|S(\varphi;B)-S(\varphi;a+B)|\leq |B\triangle (a+B)|\
\|\varphi\|_{\infty}
\leq \frac{1}{2}|S(\varphi;B)|,
$$
so that $|S(\varphi;a+B)|\geq \frac{1}{2}|S(\varphi;B)|$.  Therefore
we get
\begin{align}
  \Sigma\geq \sum_{a\in T_s(B)}|S(\varphi;a+B)|^2&\geq \frac{1}{4}
  |T_s(B)|\ |S(\varphi;B)|^2\nonumber\\
  &\geq \frac{1}{8\|\varphi\|_{\infty}}\ |S(\varphi;B)|^3-
  \frac{1}{4}|S(\varphi;B)|^2,\label{eq-basic-eighth}
\end{align}
by~(\ref{eq-almost-invariant}). Finally, combining the two bounds, we
obtain
$$
\frac{1}{8\|\varphi\|_{\infty}} \ |S(\varphi;B)|^3\leq |B|\ |D|\
\|\varphi\|_2^2+|B|^2\max_{b\notin D}|
\corr{\varphi}{b}|+\frac{1}{4}\|\varphi\|_{\infty}^2|B|^2 ,
$$
which gives the result.
\end{proof}

\begin{remark}
  By the Plancherel formula, the correlations sums $\corr{\varphi}{a}$
  have a dual formulation in terms of the Fourier transform
$$
\hat{\varphi}(\psi)=\frac{1}{\sqrt{|A|}}\sum_{x\in
  A}\varphi(x)\psi(x),
$$
defined on the dual group $\hat{A}$ of $A$: we have
\begin{equation}\label{eq-plancherel}
\corr{\varphi}{a}=\sum_{\psi\in \hat{A}}{\hat{\varphi}(\psi)
  \overline{\psi(a)\hat{\varphi}(\psi)}}= \sum_{\psi\in
  \hat{A}}{|\hat{\varphi}(\psi)|^2\ \overline{\psi(a)}}
\end{equation}
\par
This remark may be useful for special functions $\varphi$ for which
$|\hat{\varphi}|^2$ is well understood (see
Section~\ref{sec-special-case}). It is also interesting (dually) when
trying to apply the method to the dual group $\hat{A}$ which is
(non-canonically) isomorphic to $A$, since it reduces the correlation
sums to sums over $A$.
\end{remark}

\begin{remark}
% (1)   This proof does not not actually use the assumption that $A$ is
%   abelian. Thus the result holds without changes (except notational)
%   for any finite group. It is however unclear what applications might
%   exist in the case of non-abelian groups, due to the paucity of
%   examples of almost invariants subsets in such groups.
% \par
(1) We call the method ``sliding sum'' because of the intuitive
picture where we shift the graph of $\varphi$ by additive
translations, and observe that the sums small shifts of $I$ do not
differ too much from the original one.
\par
(2) The set $D$ is meant as containing the ``diagonal'' contributions.
It will contain $0$, but might in some cases be a bit larger. In
extending the method to higher dimensions, for instance, the dichotomy
introduced between shifts by $a\in D$ and $a\notin D$ might not be
sufficient to obtain a good bound. It might then be necessary to use a
finer ``stratification'' of the possible estimates for
$\corr{\varphi}{a}$. We will not pursue such situations here, but we
hope to come back to it later, in contexts involving trace functions
in more than one variable.
\end{remark}

\section{Abstract application}\label{sec-abstract}

We continue in a rather general setting before restricting our
attention to trace functions modulo primes.  We define:

\begin{definition}[Condition $\mathcal{H}(c)$]
  Let $A$ be a finite abelian group and let $\varphi\,:\, A\rightarrow
  \Cc$ be a function on $A$. For a real number $c\geq 1$, we say that
  \emph{$\varphi$ satisfies $\mathcal{H}(c)$} if
\par
(i) We have $\|\varphi\|_{\infty}\leq c$;
\par
(ii) There exists a subset $D\subset A$ with $|D|\leq c$ such that
\begin{equation}\label{eq-small}
  |\corr{\varphi}{a}|\leq c\sqrt{|A|}
\end{equation}
for $a\notin D$.
\end{definition}

The idea of this definition is that, except for special values of $a$
(the ``diagonal''), $\varphi$ should not correlate significantly with
its additive translate by $a$; of course $D$ should contain $0$, but
one can allow some more exceptional shifts. Note that this is a
property of $\varphi$ only, and not of any subset of $A$ on which we
might want to sum its values.
\par
The main estimate of Theorem~\ref{central} gives immediately:

\begin{corollary}\label{cor-concrete}
  Let $A=\Zz/m\Zz$ for some $m\geq 1$, and let $c\geq 1$ be a
  parameter. For any interval $I\subset A$ and any function $\varphi$
  on $A$ satisfying $\mathcal{H}(c)$, we have
$$
|S(\varphi;I)|\leq 2
c^{4/3}\left(|A|^{1/3}|I|^{1/3}+2|A|^{1/6}|I|^{2/3}\right),
$$
and if $|I|>\sqrt{|A|}$, we have
\begin{equation}\label{eq-better-1}
  |S(\varphi;I)|\leq 6
  c^{4/3}|I|\Bigl(\frac{|A|^{1/2}}{|I|}\Bigr)^{1/3}.
\end{equation}
% \par
% In particular, if $I$ is an interval in $\Zz/m\Zz$ and $\varphi$
% satisfies $\mathcal{H}(c)$ on $\Zz/m\Zz$, then we have
% \begin{equation}\label{eq-better0}
% |S(\varphi;I)|\leq 2c^{4/3}(m^{1/3}|I|^{1/3}+2m^{1/6}|I|^{2/3}),
% \end{equation}
% and if $|I|>\sqrt{m}$, we have
% \begin{equation}\label{eq-better}
%   |S(\varphi;I)|\leq 6c^{4/3}|I|\Bigl(\frac{m}{|I|^2}\Bigr)^{1/6}.
% \end{equation}
\end{corollary}

In comparison with~(\ref{eq-fourier}) (with $m=p$), the
bound~(\ref{eq-better-1}) replaces an estimate in terms of the
supremum of the Fourier transform with one for an ``almost'' supremum
of the additive correlation sums $\corr{\varphi}{a}$. It is
interesting to note that, in contrast with the Fourier technique, our
method is \emph{non-linear} in terms of the function $\varphi$ (a
similar feature appears in the correlation sums in~\cite{FKM1}).
\par
For fixed $c$, the estimate~\eqref{eq-better-1} is non-trivial as long
as
$$
|I|\gg \sqrt{m},
$$
where the implied constant depends on $c$, and the point of the result
is that this range of uniformity goes beyond that of the classical
completion estimates.
\par
On the other hand, if we consider functions $\varphi$ such that
$\|\hat{\varphi}\|_{\infty}\ll 1$, the Fourier estimate
\eqref{eq-fourier} is stronger than \eqref{eq-better-1} as soon as
$|I|\gg m^\frac{1}{2} (\log m)^{3/2}$.

\section{Sums over generalized arithmetic progressions}
\label{sec-gap}

This section is essentially independent of the remainder of the paper
and may be skipped in a first reading. 
\par
Let $A$ be a finite abelian group.  For a fixed integer $k\geq 1$,
recall (see~\cite[p. xii]{tao-vu}) that a \emph{$k$-dimensional
  (proper) generalized arithmetic progression} $B\subset A$ is a set
of elements of the form
\begin{equation}\label{eq-gap}
b=a_0+n_1a_1+\cdots +n_ka_k
\end{equation}
where $(a_0,\ldots,a_k)\in A^{k+1}$ and $n_i$ is in some interval
$I_i$ of integers of length $|I_i|\geq 2$, and if furthermore (this is
the meaning of being ``proper'') this representation of any $b\in B$
is unique.
\par
If $B\subset \Zz/m\Zz$ is a proper $k$-dimensional generalized
arithmetic progression, Shao~\cite{shao} has shown that the $L^1$-norm
of the Fourier transform of the characteristic function of $B$ is $\ll
(\log m)^k$, where the implied constant depends only on $k$. Thus, the
completion estimate~(\ref{eq-fourier}) gives a generalized
Polya-Vinogradov estimate of the type
\begin{equation}\label{eq-pv-gap}
\sum_{x\in B}\varphi(x)\ll \|\hat{\varphi}\|_{\infty} \sqrt{m}(\log
m )^k,
\end{equation}
where the implied constant depends only on $k$, which is non-trivial
as soon as $|B|\gg \sqrt{m}(\log m)^k$, for functions with bounded
Fourier transforms (see Remark~\ref{rm-general} for examples.)
\par
We will adapt the sliding-sum method to prove an estimate for sums
over generalized arithmetic progressions which is non-trivial in many
cases when the size of $B$ is slightly larger than
$\sqrt{m}$, thus bridging the gap between this range and the
completion range. For simplicity, we only consider the problem for
functions satisfying Condition $\mathcal{H}(c)$ for some $c\geq 1$. 

\begin{theorem}[Sums over generalized arithmetic progressions]\label{th-gap}
Let $k\geq 1$ be an integer and let $c\geq 1$ be a real parameter. 
For $m\geq 1$ an integer, let $A=\Zz/m\Zz$, and let $B\subset A$ be a
proper generalized arithmetic progression of dimension $k$ such that
$|B|\geq \sqrt{m}$.
\par
Then, for any $\varphi\,:\, A\lra \Cc$ satisfying Condition
$\mathcal{H}(c)$, we have
\begin{equation}\label{eq-gap-goal}
S(\varphi;B)\ll
|B|^{1-1/(k+2)}m^{1/(2(k+2))}=|B|\Bigl(\frac{\sqrt{m}}{|B|}\Bigr)^{1/(k+2)},
\end{equation}
where the implied constant depends only on $k$ and $c$.
\end{theorem}

As in the case of Theorem~\ref{central}, the estimate is non-trivial
as soon as $|B|\geq \alpha \sqrt{m}$ for some $\alpha$ depending on
$k$ and $c$.

\begin{proof}
We will use induction on $k\geq 1$, but we begin by a general argument
to derive the base case $k=1$ from scratch instead of appealing to the
previous result.
\par
Let $B\subset \Zz/m\Zz$ be a proper generalized arithmetic progression
of dimension $k$, formed with integers~(\ref{eq-gap}), where $n_i\in
I_i$. It will be convenient to write
$$
|B|=\beta \sqrt{m}
$$
so that $\beta\geq 1$, and we may write $\beta=\beta(B)$ when the set
under consideration changes.
\par
Let $T\geq 1$.  We distinguish two cases in trying to bound
$|S(\varphi;B)|$.
\par
(1) If
\begin{equation}\label{eq-small-bound}
  |S(\varphi;B)|\leq \frac{|B|\|\varphi\|_{\infty}}{T}\ll \frac{|B|}{T},
\end{equation}
we will just use this estimate (and thus $T$ should be chosen to
ensure that it gives the required result, but we do not fix its value
immediately in order to clarify the argument).
\par
(2) Otherwise, we have
\begin{equation}\label{eq-large-sum}
  |S(\varphi;B)|> \frac{|B|\|\varphi\|_{\infty}}{T},
\end{equation}
and we proceed by sliding sums, comparing upper and lower bounds for
$$
\Sigma=\sum_{a\in A}{\Bigl|\sum_{x\in B}\varphi(x+a)\Bigr|^2}
$$
as before. We obtain immediately the upper-bound
\begin{equation}\label{!}
\Sigma \ll m|B|+m^{1/2}|B|^2\ll m^{1/2}|B|^2
\end{equation}
since $\varphi$ satisfies $\mathcal{H}(c)$ and $|B|\geq \sqrt{m}$,
where the implied constant depends only on $c$.
\par
On the other hand, for an element 
$$
a=\sum_{i=1}^k \lambda_i a_i
$$
with $|\lambda_i|\leq |I_i|/2$ for all $i$, we see that
\begin{equation}\label{**}
|S(\varphi;B+a)|\geq |S(\varphi;B)|-
2k\max_i\Bigl(\frac{\|\varphi\|_{\infty}|B||\lambda_i|}{|I_i|}\Bigr).
\end{equation}
\par
We now distinguish two possibilities concerning the size of the
intervals defining $B$. We select
$T=(|B|/m^{1/2})^{1/(k+2)}=\beta^{1/(k+2)}$, and we assume first that,
for all $i$, we have
$$
|I_i|\geq 4kT.
$$
\par
Note that, if $k=1$, there is only one interval involved and
$|I_1|=\beta\sqrt{m}\geq T=\beta^{1/3}$, so this assumption is always
valid when $k=1$.
\par 
Taking all
$$
|\lambda_i|\leq \frac{|I_i|}{4kT},
$$
for $1\leq i\leq k$, we obtain $\gg \frac{|B|}{T^k}$ distinct shifts
for which
$$
|S(\varphi;B+a)|\geq |S(\varphi;B)|-
\frac{2k\|\varphi\|_{\infty}|B|}{4kT}\geq \frac{1}{2}|S(\varphi;B)|
$$
by~(\ref{eq-large-sum}) and \eqref{**}. Hence we have the lower bound
$$
\Sigma\gg \frac{|B|}{T^k}|S(\varphi;B)|^2,
$$
where the implied constant depends only on $k$. Comparing with \eqref{!}, we obtain
\begin{equation}\label{x}
|S(\varphi;B)|^2\ll T^{k}m^{1/2}|B|
\end{equation}
where the implied constant depends on $k$ and $c$.
\par
With our choice of $T$, we have
$$
\frac{|B|}{T}=T^{k/2}m^{1/4}|B|^{1/2}=|B|\beta^{-1/(k+2)},
$$
and therefore, by \eqref{x},
$$
S(\varphi;B)\ll |B|^{1-1/(k+2)}m^{1/(2(k+2))}
$$
in this case, as claimed. In particular, this establishes the result
when $k=1$.
\par
We now proceed to conclude using induction on $k$. Since the case
$k=1$ is established, we may assume that we consider $k\geq 2$, and
that the estimate of the theorem is valid for progressions of
dimension $\leq k-1$.
\par
We consider again $T=\beta^{1/(k+2)}$, and we assume that the
intervals are ordered in such a way that
$$
|I_1|\leq |I_2|\leq \cdots \leq |I_k|,
$$
and we are now assuming that for some $j$ with $1\leq j\leq k$, we
have
$$
|I_1|\leq \cdots \leq |I_j|<4kT\leq |I_{j+1}|,
$$
\par
Note that
$$
L=\prod_{i\leq j}|I_i|\ll T^j=\beta^{j/(k+2)}
$$
and therefore
\begin{equation}\label{eq-beta}
\frac{|B|}{L}=\frac{\beta m^{1/2}}{L}\gg m^{1/2}\beta^{1-j/(k+2)},
\end{equation}
which implies in particular that $L<|B|$, i.e., that $j<k$, for $m$ sufficiently large in terms of $k$. 
\par
The set $B$ decomposes into a disjoint union of $L$ proper generalized
arithmetic progressions (noted $B_{a}$)  of dimension $k-j<k$, each of size $|B|/L\geq
\sqrt{m}$. Over each of these, the function $\varphi$ satisfies the
Condition $\mathcal{H}(c)$.
\par
By induction, over each subprogression $B_a$, we have
$$
\sum_{x\in B_a}\varphi(x)\ll |B_a|
\beta(B_a)^{-1/(k-j+2)}=\frac{|B|}{L} \beta(B_a)^{-1/(k-j+2)},
$$
where the implied constant depends only on $k$ and
$c$. By~(\ref{eq-beta}), since $|B_a|=|B|/L$, we have $\beta(B_a)\geq
\beta^{1-j/(k+2)}$ (where $\beta=\beta(B)$) so that
$$
\beta(B_a)^{1/(k-j+2)}\geq \beta^{1/(k+2)},
$$
and hence 
$$
S(\varphi;B_a)\ll \frac{|B|}{L}\beta^{-1/(k+2)},
$$
for each subprogression. Summing over the $L$ progressions $B_{a}$ of
dimension $k-j$, we get
$$
S(\varphi;B)\ll |B|\beta^{-1/(k+2)},
$$
as desired.
\end{proof}

\section{Trace functions: the additive case}\label{alggeo}

The trace functions of suitable $\ell$-adic sheaves modulo primes,
which we have studied, and used in applications, in a number of recent
works (\cite{FKM1,FKM2,FKM3,counting-sheaves,gowers-norms}), provide
many examples of functions on $A=\Fp=\Zz/p\Zz$ satisfying
$\mathcal{H}(c)$ for $c$ bounded independently of $p$.
\par
To state this fact in a precise way, we recall some standard
conventions. For any prime $\ell$, we fix an isomorphism $\iota\,:\,
\bar{\Qq}_{\ell}\simeq \Cc$, and we use it implicitly as an
identification for any $\ell$-adic number. An \emph{isotypic Fourier
  sheaf} modulo a prime $p$ is defined to be a constructible
middle-extension $\ell$-adic sheaf $\sheaf{F}$ on $\mathbf{A}^1_{\Fp}$
for some $\ell\not=p$, which is pointwise pure of weight $0$,
geometrically isotypic, and of Fourier type in the sense of Katz,
i.e., its geometric irreducible component is not an Artin-Schreier
sheaf $\sheaf{L}_{\psi}$ for some additive character $\psi$. 
\par
The conductor of a middle-extension $\ell$-adic sheaf on
$\mathbf{A}^1_{\Fp}$ is defined to be
$$
\cond(\sheaf{F})=\rank(\sheaf{F})+n(\sheaf{F})+\sum_{x\in
  S(\sheaf{F})}\swan_x(\sheaf{F}),
$$
where $S(\sheaf{F})\subset \mathbf{P}^1(\bar{\Ff}_p)$ is the set of
singularities of $\sheaf{F}$, $n(\sheaf{F})$ is the cardinality of $S$
and $\swan_x$ denotes the Swan conductor at such a singularity. Thus
$\cond(\sheaf{F})$ is a positive integer measuring the complexity of
$\sheaf{F}$.

\begin{example}
Let 
$$
\varphi(n)=\chi(f(n))e\Bigl(\frac{g(n)}{p}\Bigr)
$$
where $\chi$ is a Dirichlet character modulo $p$ and $f, g\in \Qq(X)$
are fixed rational functions. Then, for all primes $p$ such that $f$
and $g$ modulo $p$ satisfy one of the conditions described in the
introduction, the function $\varphi$ is a trace function associated to
a middle-extension sheaf $\sheaf{F}$ with
$$
\cond(\sheaf{F})\ll \deg(f_1)+\deg(f_2)+\deg(g_1)+\deg(g_2),
$$
where the implied constant is absolute.
\end{example}

Given a middle-extension $\sheaf{F}$ modulo $p$, we denote by
$t_{\sheaf{F}}$ its \emph{trace function}, which is the function
$$
t_{\sheaf{F}}\,:\, \Fp\rightarrow \Cc
$$
defined by
$$
t_{\sheaf{F}}(x)=\iota(\frtr{\sheaf{F}}{\Fp}{x}),
$$
the trace of the action of the Frobenius of $\Fp$ acting on the stalk
at $x\in \mathbf{A}^1(\Fp)$ of $\sheaf{F}$. It is known that
$$
|t_{\sheaf{F}}(x)|\leq \cond(\sheaf{F})
$$
for all $x\in\Fp$ (for $x$ not a singularity of the sheaf, this
follows from the fact that the trace is the sum of $\rank(\sheaf{F})$
complex numbers of modulus $\leq 1$, and for singularities, it is a
consequence of the fact that $\sheaf{F}$ is a middle-extension and a
result of Deligne.)
\par
The crucial fact we use to control correlations is the following
version of Deligne's Riemann Hypothesis:

\begin{theorem}\label{th-rh}
  Let $p$ be a prime number, $c\geq 1$, and let $\sheaf{F}_1$ and
  $\sheaf{F}_2$ be two isotypic Fourier sheaves modulo $p$ with
  conductor $\leq c$. If the geometric irreducible components of
  $\sheaf{F}_1$ and $\sheaf{F}_2$ are not isomorphic, then we have
\begin{gather*}
\Bigl|\sum_{x\in \Fp}t_{\sheaf{F}_1}(x)\overline{t_{\sheaf{F}_2}(x)}
\Bigr|\leq 5c^3\sqrt{p},\\
\Bigl|\sum_{x\in \Fpt}t_{\sheaf{F}_1}(x)\overline{t_{\sheaf{F}_2}(x)}
\Bigr|\leq 6c^3\sqrt{p}.
\end{gather*}
\end{theorem}

\begin{proof}
  Let $U$ be a non-empty open set of $\Aa^1_{\Fp}$ where $\sheaf{F}_1$
  and $\sheaf{F}_2$ are both lisse; one can find such a $U$ with
  $|\Fp-U(\Fp)|\leq 2c$, which we assume to be true. We then have
$$
\Bigl|\sum_{x\in \Fp}t_{\sheaf{F}_1}(x)\overline{t_{\sheaf{F}_2}(x)}
\Bigr|\leq \Bigl|\sum_{x\in
  U(\Fp)}t_{\sheaf{F}_1}(x)\overline{t_{\sheaf{F}_2}(x)} \Bigr|
+|\Fp-U(\Fp)|c^2
$$
since $|t_{\sheaf{F}_i}(x)|\leq c$ for $i=1$, $2$ and all $x\in\Fp$.
\par
By the quasi-orthonormality result of~\cite[Lemma
3.5]{counting-sheaves} (or its obvious extension to geometrically
isotypic sheaves), which follows from the Riemann Hypothesis over
finite fields, we have 
$$
\Bigl|\sum_{x\in
  U(\Fp)}{t_{\sheaf{F}_1}(x)\overline{t_{\sheaf{F}_2}(x)}}\Bigr| \leq
3c^3\sqrt{p},
$$
and the first bound follows.
\par
For the sum over $\Fpt$, we just write
$$
\Bigl|\sum_{x\in \Fpt}t_{\sheaf{F}_1}(x)\overline{t_{\sheaf{F}_2}(x)}
\Bigr|
\leq \Bigl|\sum_{x\in
  \Fp}t_{\sheaf{F}_1}(x)\overline{t_{\sheaf{F}_2}(x)} 
\Bigr|+|t_{\sheaf{F}_1}(0)\overline{t_{\sheaf{F}_2}(0)}|
\leq 5c^3\sqrt{p}+c^2.
$$
\end{proof}

We can now apply the sliding sum method to trace functions:

\begin{proposition}\label{pr-no-correlation}
  Let $p$ be a prime number, and let $\sheaf{F}$ be an isotypic
  Fourier sheaf modulo $p$ with conductor $c$. Then the trace function
  $t_{\sheaf{F}}$ satisfies $\mathcal{H}(5c^3)$. In particular, we
  have
$$
\Bigl|\sum_{x\in I}t_{\sheaf{F}}(x)\Bigr|\leq
18c^4(p^{1/3}|I|^{1/3}+2p^{1/6}|I|^{2/3})
$$
for any interval $I\subset \Fp$, and
$$
\Bigl|\sum_{x\in I}t_{\sheaf{F}}(x)\Bigr|\leq
54c^{4}|I|\Bigl(\frac{\sqrt{p}}{|I|}\Bigr)^{1/3}
$$
for any interval $I$ in $\Fp$ with $|I|>\sqrt{p}$.
\end{proposition}

This proposition is a more precise form of Theorem~\ref{th-slide},
and completes the proof of that result.

\begin{proof}%%
  Since $|t_{\sheaf{F}}(x)|\leq c$ for all $x\in\Fp$, the first
  condition in $\mathcal{H}(5c^3)$ certainly holds, and we need to
  consider the correlation sums. For $a\in \Fp$, the function
  $x\mapsto t_{\sheaf{F}}(x+a)$ is the trace function of the sheaf
  $[+a]^*\sheaf{F}$, which is also an isotypic Fourier sheaf, and
  which has the same conductor as $\sheaf{F}$. By Theorem~\ref{th-rh},
  we have
$$
\Bigl|\sum_{x\in\Fp}{t_{\sheaf{F}}(x)\overline{t_{\sheaf{F}}(x+a)}}
\Bigr|\leq 5c^3\sqrt{p}
$$
unless the geometrically irreducible component of $\sheaf{F}$ (say
$\sheaf{G}$) is geometrically isomorphic to that of $[+a]^*\sheaf{F}$,
which is easily seen to be $[+a]^*\sheaf{G}$. Now suppose this is the
case for some $a\not=0$. Then it follows from~\cite[Lemma 5.4
(2)]{gowers-norms} (applied to $\sheaf{G}$, with $d=0$) that $c\geq
p$. But in that case we have the trivial bound
$$
|\corr{t_{\sheaf{F}}}{a}|\leq c^2p\leq c^3\leq c^3\sqrt{p}.
$$
\par
This means that we always have
$$
|\corr{t_{\sheaf{F}}}{a}|\leq 5c^3\sqrt{p}
$$
for all $a\not=0$, and hence we can take $D=\{0\}$ in checking
$\mathcal{H}(5c^3)$. The final estimates are then just the
applications of Corollary~\ref{cor-concrete}, since $2\cdot
5^{4/3}\leq 18$ and $6\cdot 5^{4/3}\leq 54$.
\end{proof}

\begin{remark}[Polya-Vinogradov bound for trace
  functions]\label{rm-general} 
  As already mentioned in the introduction, trace functions also
  satisfy a very general analogue of the Polya-Vino\-gra\-dov
  bound~(\ref{eq-polya-vino-chars}). More precisely, recall
  (see~\cite[Lemma 8.1, Prop. 8.2]{FKM1}) that if $\sheaf{F}$ is an
  isotypic Fourier sheaf, there exists a Fourier transform sheaf
  $\sheaf{G}$, defined by Deligne, such that
$$
t_{\sheaf{G}}(t)=-\frac{1}{\sqrt{p}}\sum_{x\in
  \Fp}{t_{\sheaf{F}}(x)e\Bigl(\frac{tx}{p}\Bigr)}=
-\hat{t}_{\sheaf{F}}(t)
$$
for all $t\in\Fp$. This sheaf is still an isotypic Fourier sheaf and
has conductor $\cond(\sheaf{G})\leq 10\cond(\sheaf{F})^2$, and
therefore, for such a sheaf $\sheaf{F}$, we have
$$
\|\hat{t}_{\sheaf{F}}\|_{\infty}\leq 10\cond(\sheaf{F})^2,
$$
so that~(\ref{eq-fourier}) gives
$$
\sum_{n\in I}t_{\sheaf{F}}(n)\ll \cond(\sheaf{F})^2\sqrt{p}(\log p)
$$
for any interval $I$ in $\Fp$, where the implied constant is
absolute. (The first cases of such bounds for sheaves which are not of
rank $1$ are found in~\cite{michel}.)
\par
Similarly (see~(\ref{eq-pv-gap})), Shao's result~\cite{shao} gives a
bound
$$
\sum_{n\in B}\varphi(n)\ll \cond(\sheaf{F})^2 \sqrt{p}(\log
p )^k
$$
if $B\subset \Fp$ is a proper $k$-dimensional generalized arithmetic
progression. 
\end{remark}

% Then there exists $\theta\in\Cc$ with modulus $1$ such that
% $$
% t_{\sheaf{F}}(x+a)=\theta t_{\sheaf{F}}(x)
% $$
% for all $x\in \Fp$. Taking $x=0$ and iterating, it follows that
% $t_{\sheaf{F}}(x)=\theta^xt_{\sheaf{F}}(0)$ for integers $0\leq x<p$,
% in particular $|t_{\sheaf{F}}|$ is constant. If it were non-zero,
% since $\sheaf{F}$ is not geometrically trivial (it is not 

We can now prove our equidistribution corollary.

\begin{proof}[Proof of Corollary~\ref{cor-rational}]
  (1) We can certainly assume that $\beta(n)<n^{1/2}$ for all $n$.  By
  the Weyl criterion, we must show that, for any fixed integer
  $h\not=0$, and for the interval $I_p$, the sums
$$
\frac{1}{|I_p|}\sum_{n\in I}e\Bigl(\frac{hf(n)}{p}\Bigr)
$$
tend to $0$ as $p\ra +\infty$. For a given $p$, and a suitable
$\ell$-adic non-trivial additive character $\psi$ of $\Fp$, we
consider the rank $1$ sheaf
$$
\sheaf{F}=\sheaf{L}_{\psi(hf(X))}
$$
which has trace function
$$
t_{\sheaf{F}}(x)=e\Bigl(\frac{hf(x)}{p}\Bigr)
$$
for all $x\in\Fp$. This is a middle-extension sheaf modulo $p$,
geometrically irreducible, pointwise pure of weight $0$. For $p$ large
enough so that $hf(X)$ is not a polynomial of degree $\leq 1$, this
sheaf is a Fourier sheaf. Its conductor satisfies
$$
\cond(\sheaf{F})\leq 1+(1+\deg(f_2))+\sum_{x\text{ pole of } f_2}
\mathrm{ord}_{x}(f_2)+\deg(f_1)\ll 1
$$
for all $p$ large enough (the first $1$ is the rank, the singularities
are at most at poles of $f_2$ and at $\infty$, the Swan conductor at a
pole of $f_2$ is at most the order of the pole, and at infinity it is
at most the order of the pole of $f$ at infinity, which is at most the
degree of $f_1$). Hence, by Proposition~\ref{pr-no-correlation}, for
some $c\geq 1$ independent of $p$, the trace function $t_{\sheaf{F}}$
satisfies $\mathcal{H}(c)$ for all $p$ large enough. By
Corollary~\ref{cor-concrete}, we get
$$
\frac{1}{|I_p|}\sum_{n\in
  I}e\Bigl(\frac{hf(n)}{p}\Bigr)=\frac{1}{|I_p|} S(t_{\sheaf{F}};I)
\ll \Bigl(\frac{\sqrt{p}}{|I_p|}\Bigr)^{1/3}\ll \beta(p)^{-1/3}\ra 0
$$
by assumption. 
\par
(2) Using the Weyl criterion, and keeping some notation from (1), it
is enough to show that for any fixed $d\geq 1$, we have
$$
\lim_{p\rightarrow +\infty} \frac{1}{|I_p|}\sum_{x\in
  I}U_d(2\cos\theta_p(x))= \frac{1}{|I_p|}\sum_{x\in I}U_d
\Bigl(\frac{S(x,1;p)}{\sqrt{p}}\Bigr) =0,
$$
where $U_d\in\Zz[X]$ is the Chebychev polynomial such that
$$
U_d(2\cos\theta)=(\sin\theta)^{d+1}/(\sin\theta).
$$
\par
By the theory of Deligne and Katz of Kloosterman
sheaves~\cite{katz-gskm}, the function
$$
\varphi(x)=U_d(2\cos\theta_p(x))
$$
is the trace function of a geometrically irreducible sheaf (the
symmetric $d$-th power of the rank $2$ Kloosterman sheaf) of rank
$d+1\geq 2$ on the affine line over $\Fp$, and this sheaf has
conductor bounded by a constant depending only on $d$, and not on
$p$. It is therefore a Fourier sheaf with trace function
satisfying~$\mathcal{H}(c)$ for some $c$ depending only on $d$, and
hence the desired limit holds again by a direct application of
Proposition~\ref{pr-no-correlation}. (See also, e.g.,~\cite[\S
10.3]{FKM1} for such facts about Kloosterman sheaves.)
\end{proof}

A somewhat similar application is the following:

\begin{proposition}[Polynomial residues]\label{pr-pol-residues}
  Let $\beta$ be a function defined on integers such that $1\leq
  \beta(n)\ra +\infty$ as $n\ra +\infty$. Let $f\in \Zz[X]$ be a
  non-constant monic polynomial. For all primes $p$ large enough,
  depending on $f$ and $\beta$, and for any interval $I_p$ modulo $p$
  of size $|I_p|\geq p^{1/2}\beta(p)$, there exists $x\in I_p$ such
  that $x=f(y)$ for some $y\in \Fp$. In fact, denoting by $P$ the set
  $f(\Fp)$ of values of $f$, the number of such $x$ is $\sim \delta_f
  |I_p|$ as $p\ra +\infty$, where $\delta_f=|P|/p$.
\end{proposition}

Here again, the interest of the result is when $\beta(n)$ is smaller
than $\log n$. However, it seems likely that this distribution
property should be true for much shorter intervals.

\begin{proof}
  Let $\varphi$ be the characteristic function of the set $P$ of
  values $f(y)$ for $y\in\Fp$. We must show that, for $p$ large
  enough, we have
$$
\sum_{x\in I_p}\varphi(x)\sim \delta_f |I_p|
$$
(which in particular implies that the left-hand side is $>0$ for $p$
large enough.)
\par
By~\cite[Prop. 6.7]{FKM2}, if $p$ is larger than $\deg(f)$, there
exists a decomposition
$$
\varphi(x)=\sum_i c_i \varphi_i(x)
$$
where the number of terms in the sum and the $c_i$ are bounded in
terms of $\deg(f)$ only, and where $\varphi_i$ is the trace function
of a \emph{tame} isotypic sheaf $\sheaf{F}_i$ with conductor bounded
in terms of $\deg(f)$ only. Moreover, $\sheaf{F}_1$ is the trivial
sheaf with trace function equal to $1$, all others are geometrically
non-trivial, and
$$
c_1=\delta_f+O(p^{-1/2}),
$$
where $\delta_f=|P|/p$ and the implied constant depends only on
$\deg(f)$. Note that $\delta_f\gg 1$ for primes $p>\deg(f)$. 
\par
From this, we obtain
$$
\sum_{x\in I_p}\varphi(x)=c_1|I_p| +\sum_{i\not=1} c_i
S(t_{\sheaf{F}_i};I_p)= \delta_f|I_p| +\sum_{i\not=1} c_i
S(t_{\sheaf{F}_i};I_p)+O(p^{-1/2}|I_p|).
$$
\par
Since the $\sheaf{F}_i$, for $i\not=1$, are tame and non-trivial, they
are isotypic Fourier sheaves, and hence by
Proposition~\ref{pr-no-correlation}, we get
$$
S(t_{\sheaf{F}_i};I_p)\ll
|I_p|\Bigl(\frac{\sqrt{p}}{|I_p|}\Bigr)^{1/3} \ll |I_p|
\beta(p)^{-1/3},
$$
for each $i\not=1$, where the implied constant depends only on
$\deg(f)$. Hence we obtain
$$
\sum_{x\in I_p}\varphi(x)\sim \delta_f|I_p|
$$
uniformly for $p>\deg(f)$, since $\beta(p)\ra+\infty$, which gives the
result.
\end{proof}

\begin{remark}
Combining the first part of Proposition~\ref{pr-no-correlation} with
Theorem~\ref{th-gap} (instead of Corollary~\ref{cor-concrete}), we
obtain an analogue of Theorem~\ref{th-slide}
where the interval $I$ is replaced by a $k$-dimensional generalized
arithmetic progressions $B\subset \Fp$, with $k$ fixed, namely
$$
\sum_{x\in B}\varphi(x)\ll
|B|\Bigl(\frac{p^{1/2}}{|B|}\Bigr)^{1/(k+2)},
$$
where the implied constant depends on $c$ and $k$.
\par
Then, we derive immediately the analogues of
Corollary~\ref{cor-rational} and Proposition~\ref{pr-pol-residues}
where the intervals are replaced by $k$-dimensional generalized
arithmetic progressions $B\subset \Fp$ such that
$|B|=p^{1/2}\beta(p)$, where again $k$ is fixed.
\end{remark}

\section{Trace functions: the multiplicative case}\label{sec-mult}

We consider now a different application of the result of
Section~\ref{sec-abstract}: for a prime $p$, we look at the values of
trace functions modulo $p$ on the multiplicative group $A=\Fpt\simeq
\Zz/(p-1)\Zz$.  Fixing a generator $g$ of $A$, we are now looking at
sums over geometric progressions $xg^n$ for $n$ in some interval $I$
in $\Zz/(p-1)\Zz$.  Such sums are considered in~\cite[Ch. 1, \S
7]{korobov}.
\par
We will use the notation and terminology of the previous section, but
to avoid confusion we write $\tau_{\sheaf{F}}$ for the restriction of
the trace function of a sheaf $\sheaf{F}$ to $\Fpt$. The
multiplicative analogue of Proposition~\ref{pr-no-correlation} is
then:

\begin{proposition}\label{pr-no-correlation-mult}
  Let $p$ be a prime number, and let $\sheaf{F}$ be an isotypic sheaf
  modulo $p$ with conductor $c$ with geometric irreducible component
  not isomorphic to a Kummer sheaf $\sheaf{L}_{\chi}$ for some
  multiplicative character $\chi$. Then the trace function
  $\tau_{\sheaf{F}}$ satisfies $\mathcal{H}(6c^3)$ for the group
  $\Fpt$. In particular, if $g$ is a generator of $\Fpt$, we have
$$
\Bigl|\sum_{n\in I}\tau_{\sheaf{F}}(g^n)\Bigr|\leq
66c^{4}|I|\Bigl(\frac{\sqrt{p-1}}{|I|}\Bigr)^{1/3},
$$
for any interval $I$ in $\Zz/(p-1)\Zz$ with $|I|>\sqrt{p-1}$.
\end{proposition}

\begin{proof}
  Fix as above a generator $g$ of $\Fpt$.  For $a\in \Zz/(p-1)\Zz$,
  the correlation sums are now given by
$$
\corr{\tau_{\sheaf{F}}}{a}=
\sum_{n\in\Zz/(p-1)\Zz}\tau_{\sheaf{F}}(g^n)
\overline{\tau_{\sheaf{F}}(g^{a+n})}
=\sum_{x\in\Fpt}\tau_{\sheaf{F}}(x)\overline{\tau_{\sheaf{F}}(xy)}
$$
where $y=g^a$. The function $x\mapsto \tau_{\sheaf{F}}(xy)$ is then
the restriction to $\Fpt$ of the trace function of the sheaf $[\times
y]^*\sheaf{F}$, which is again an isotypic sheaf, and which has the
same conductor as $\sheaf{F}$. By the second bound in
Theorem~\ref{th-rh}, we get
\begin{equation}\label{eq-good}
  |\corr{\tau_{\sheaf{F}}}{a}|\leq 6c^3\sqrt{p}
\end{equation}
unless the geometrically irreducible component $\sheaf{G}$ of
$\sheaf{F}$ is geometrically isomorphic to $[\times
y]^*\sheaf{G}$. Now let
$$
\tilde{D}=\{y\in \bar{\Ff}_p^{\times}\,\mid\, \sheaf{G}\simeq [\times
y]^*\sheaf{G} \}
$$
(where $\simeq$ means geometric isomorphism.) This is a subgroup of
$\bar{\Ff}_p^{\times}$, and in fact, by~\cite[Prop. 6.4]{FKM2}, it is
an algebraic subgroup of the multiplicative group. Furthermore,
by~\cite[Prop. 6.5, (2)]{FKM2}, it is a \emph{finite} subgroup under
our assumption that $\sheaf{G}$ is not a Kummer sheaf. Let then $D$ be
the intersection of $\tilde{D}$ with $\Fpt$, which is a subgroup of
$\Fpt$, such that $\sheaf{F}$ satisfies~(\ref{eq-good}) for all
$a\notin D$. Now we distinguish two cases (the argument is implicit in
the proof of~\cite[Th. 6.3]{FKM2}): (1) if $\sheaf{F}$ is not lisse on
$\mathbf{G}_m$, then all points of the $D$-orbit of a singularity
$x\in\mathbf{G}_m$ are singularities, and hence
$$
c=\cond(\sheaf{F})\geq |D|,
$$
in which case $\mathcal{H}(6c^3)$ is true; (2) if $\sheaf{F}$ is lisse
on $\mathbf{G}_m$, then it is not tamely ramified (since a tamely
ramified lisse sheaf on $\mathbf{G}_m$ is geometrically a direct sum
of Kummer sheaves) but then~\cite[Lemma 6.6]{FKM2} shows that
$\swan_{\infty}(\sheaf{F})\geq |D|$, hence $c\geq |D|$ once more.
\par
Finally, since $6\cdot 6^{4/3}\leq 66$, we see
that~(\ref{eq-better-1}) gives the bound we claim for sums over
geometric progressions.
\end{proof}

\section{Special improvements}\label{sec-special-case}

The general argument leading to Theorem~\ref{central} can be improved
very slightly in special cases, both with respect to the summation set
$B$, and with respect to the function $\varphi$. These tweaks affect
separately the upper and lower bounds for the sum
$$
\Sigma=\sum_{a\in\Fp}{|S(\varphi;a+B)|^2}.
$$
\par
We begin with the lower bound, which we can improve when $B=I$ is an
interval in $\Zz/m\Zz$.

\begin{lemma}\label{lm-interval}
  Let $m\geq 1$ be an integer and let $\varphi$ be a function on
  $A=\Zz/m\Zz$. For any $\eps>0$, and any interval $I$ in $A$, we have
$$
\sum_{a\in\Fp}{|S(\varphi;a+I)|^2}\geq
\Bigl(\frac{1}{3}-\eps\Bigr)\frac{1}{\|\varphi\|_{\infty}} |S(\varphi;I)|^3
$$
provided $|S(\varphi;I)|/\|\varphi\|_{\infty}$ is large enough in
terms of $\eps$.
\end{lemma}

The factor $1/3$ improves here the factor $1/8$ of the general
inequality~(\ref{eq-basic-eighth}).

\begin{proof}
  Let $S=|S(\varphi;I)|$ and $\nu=\|\varphi\|_{\infty}$. For an
  integer $a\in\Zz$, we have already noted that $|I\triangle
  (I+a)|\leq 2|a|$, and
$$
|S(\varphi;a+I)|\geq |S(\varphi;I)|-2|a|\nu=S-2|a|\nu
$$
so we obtain
$$
\Sigma\geq S^2+2\sum_{1\leq j\leq \sigma}{(S-2j\nu)^2}
$$
as long as
$\sigma<\min(\tfrac{m}{2},\tfrac{S}{2\nu})=\tfrac{S}{2\nu}$. The
right-hand side is equal to
$$
S^3\Bigl\{\frac{1}{S}+\nu^{-1}\times \frac{2\nu}{S} \sum_{1\leq
  j\leq\sigma}\Bigl(1-\frac{2j\nu}{S}\Bigr)^2\Bigr\}.
$$
\par
Provided $S/(2\nu)$ is large enough, the inner sum is a Riemann sum
for
$$
\int_0^1{(1-x)^2dx}=\frac{1}{3},
$$
and the result follows.
\end{proof}

% \section{Special cases}\label{sec-special-case}

We next consider special cases of functions $\varphi$ for which the
correlations $\corr{\varphi}{a}$ are known exactly, in which case the
upper-bound for $\Sigma$ can be improved.
\par
One example was already mentioned in the introduction, and is the
function on $\Fp$, for $p\geq 3$, defined by
\begin{equation}\label{quadratic}
  \varphi (x) = e \Bigl(\frac{hx^2}{p} \Bigr).
 % \text {with } p\geq 3
 %  \text{ and } p\nmid h.
\end{equation}
\par
More generally, for $p$ prime, let $f$, $g\in\Fp(X)$ be rational
functions, let $\chi$ be a multiplicative character of $\Fpt$ and
define
$$
\psi(x)=\chi(f(x))e\Bigl(\frac{g(x)}{p}\Bigr),
$$
(with the usual conventions when $x$ is a pole of $g$ or a zero or
pole of $f$). Now let $\varphi$ be the (opposite of the) \emph{Fourier
  transform} of $\psi$, i.e.
\begin{equation}\label{eq-special}
\varphi(x)=-\frac{1}{\sqrt{p}}\sum_{y\in\Fp}{\psi(y)
e\Bigl(\frac{xy}{p}\Bigr)}=-\frac{1}{\sqrt{p}}\sum_{y\in\Fp}{
\chi(f(y))e\Bigl(\frac{g(y)+xy}{p}\Bigr)}
\end{equation}
so that $\varphi$ defines a family of one-variable character sums.
\par
Note that, by a classical computation, if $p\geq 3$, $x\in \Fp$ and
$h\in\Fpt$, we have
$$
e\Bigl(\frac{hx^2}{p}\Bigr)= \frac{1}{w_p(h)}
\sum_{y\in\Fp}{e\Bigl(-\frac{\overline{4h}y^2}{p}
  \Bigr)e\Bigl(\frac{xy}{p}\Bigr)},
$$
where
$$
w_p(h)=\frac{1}{\sqrt{p}}\sum_{y\in\Fp}e\Bigl(\frac{hy^2}{p}\Bigr)
$$
is a normalized Gauss sum, and hence has modulus $1$. This shows
that~(\ref{quadratic}) is, up to a constant factor of modulus $1$
independent of $x$, a special case of this definition.
\par
The main point is the following lemma:

\begin{lemma}\label{lm-special}
  Let $p$ be a prime number and let
$$
\varphi(x)=-\frac{1}{\sqrt{p}}\sum_{y\in\Fp}{
\chi(f(y))e\Bigl(\frac{g(y)+xy}{p}\Bigr)}
$$ 
where $f\in\Fp[X]$ and $g\in\Fp[X]$ are polynomials. Assume that $f$
is a polynomial with no zero in $\Fp$ \emph{(}for instance $f=1$ or an
irreducible polynomial of degree $\geq 2$.\emph{)} 
\par
We then have $\corr{\varphi}{a}=0$ for all $a\in\Fpt$ and
$\corr{\varphi}{0}=p$. In particular, we have
$$
\sum_{a\in\Fp}{|S(\varphi;a+B)|^2}=p|B|
$$
for any subset $B\subset \Fp$.
\end{lemma}

Note that this lemma does apply to~(\ref{quadratic}) with $g=X^2$ (and
$f=1$.)

\begin{proof}
By the Plancherel formula (see~(\ref{eq-plancherel})), we have
$$
\corr{\varphi}{a}=\sum_{t\in \Fp}
{|\hat{\varphi}(t)|^2e\Bigl(\frac{at}{p}\Bigr)}= \sum_{t\in \Fp}
{|\psi(t)|^2e\Bigl(\frac{at}{p}\Bigr)},
$$
(where $\psi$ is the function~(\ref{eq-special})), and under the
assumptions of the lemma, we see that $|\psi(t)|^2=1$ for \emph{all}
$t\in \Fp$, hence the result.
\end{proof}

We can use either Lemma~\ref{lm-interval} or Lemma~\ref{lm-special} to
derive variants of Theorem~\ref{central}. We just state the
combination of the two:

\begin{theorem}\label{paticular1}
  Let $p$ be a prime number, and let $\varphi$ be a function defined
  on $\Fp$ by~\emph{(\ref{eq-special})} such that $f$ is a polynomial
  in $\Fp[T]$, and that $g$ is a polynomial with no zero in
  $\Fp$. Then, for every $\eps>0$ and every interval $I\subset \Fp$
  such that $|I|$ is large enough in terms of $\eps$, we have
\begin{equation}\label{in1}
  | S (\varphi; I) |
  \leq  (3^{1/3}+\eps)\|\varphi\|_{\infty}^{1/3} 
  |I|^{1/3}p^{1/3}.
% \leq 3^\frac{1}{3} \vert I 
%   \vert \Bigl(\frac{p}{\vert I \vert^2} \Bigr)^\frac{1}{3}.
\end{equation}
\par
In particular, for $p\geq 3$ and $h\in\Fpt$, we have
$$
\frac{1}{|I|} \Bigl|\sum_{x\in I}e\Bigl(\frac{hx^2}{p}\Bigr)\Bigr|
\leq (3^{1/3}+\eps)\Bigl(\frac{\sqrt{p}}{|I|}\Bigr)^{2/3}
$$
provided $|I|$ is large enough in terms of $\eps$.
% \par
% Furthermore, let $i_{0}$ modulo $p$ be such that the interval $I$ is
% given by
% $$
% I= \bigl[i_{0}, i_{0}+\vert I\vert\bigr[ \ \bmod p,
% $$
% and assume that, modulo $p$, we have
% \begin{equation}\label{disjoint}
%    \bigl]i_{0}- 2\cdot \vert I\vert , i_{0}
%    +2\cdot \vert I \vert \bigr[\  \cap\  \bigl]-i_{0}- 
%    2\cdot \vert I\vert ,- i_{0}+2\cdot \vert I \vert \bigr[ =\emptyset.
% \end{equation}
% \par
% Then we have
% \begin{equation}\label{in2}
%   \bigl\vert S (\varphi; I) \bigr\vert \leq (3/2)^\frac{1}{3} 
%   \vert I \vert \Bigl(\frac{p}{\vert I \vert^2} \Bigr)^\frac{1}{3}.
% \end{equation}
\end{theorem}
 
\begin{proof}
We just combine the upper and lower bounds for the sum $\Sigma$ given
by Lemmas~\ref{lm-interval} and~\ref{lm-special}, observing that if
$|S(\varphi;I)|/\|\varphi\|_{\infty}$ is too small for
Lemma~\ref{lm-interval} to apply, the resulting bound
$$
|S(\varphi;I)|\leq A(\eps)\|\varphi\|_{\infty}
$$
is stronger than~(\ref{in1}) if $|I|$ is large enough.
\end{proof}

\begin{remark}
  The exponent $2/3$ appearing in \eqref{in1} improves the exponent
  $1/3$ appearing in \eqref{eq-better-1} for instance. This is due to
  the absence of non--diagonal terms.  This result implies that the
  classical bound \eqref{eq-fourier} is better than \eqref{in1} when
  $\vert I \vert \gg p^\frac{1}{2} \log^3 p$, always in the particular
  case where $\varphi$ is defined by \eqref{quadratic}
\par
The result of Theorem \ref{paticular1} gives a non--trivial bound of
$\vert S (\varphi; I)\vert $ as soon as $\vert I\vert \geq (
3^\frac{1}{2} +o(1)) p^\frac{1}{2}$, which is rather close to the
critical length $p^{1/2}$.
% and $\vert I\vert \geq (( 3/2)^\frac{1}{2} +o(1)) p^\frac{1}{2}$.
\end{remark}

\begin{remark}\label{rm-kloos}
  The conditions described in Lemma~\ref{lm-special} are not the only
  ones for which we can prove Theorem~\ref{paticular1}. For instance,
  suppose $f$ (resp. $g$) has at worse a pole at $0$ (resp. at worse
  a pole or zero at $0$), which is the case for instance when
$$
\psi(x)=e(\bar{x}/p)
$$
for $x\in\Fpt$ and $\psi(0)=0$, in which case
$\varphi(x)=-S(x,1;p)/\sqrt{p}$ (the normalized Kloosterman sum) for
all $x\in \Fp$. Then we find that
$$
\corr{\varphi}{a}=\begin{cases}p-1&\text{ if } a=0\\
-1&\text{ otherwise}
\end{cases}
$$
which means that the sum~(\ref{defSigma}) is now equal to
$$
\sum_{x,y\in B}\corr{\varphi}{y-x}=p|B|-|B|^2\leq p|B|,
$$
from which we see that the result of Theorem~\ref{paticular1}
holds. If $I$ is an interval in $\Fp$, we get for instance
$$
\Bigl|\frac{1}{|I|}\sum_{x\in I}\frac{S(x,1;p)}{\sqrt{p}}\Bigr| \leq
(3^{1/3}+\eps) \Bigl(\frac{\sqrt{p}}{|I|}\Bigr)^{2/3}
$$
when $|I|$ is large enough in terms of $\eps$.
%% (since $32^{1/3}\leq 3.2$).
\end{remark}

We will finish with a multiplicative special case one, as in \S
\ref{sec-mult}, inspired by~\cite[Ch. 1, \S 7]{korobov}. We let
\begin{equation}\label{korobov}
\varphi (n) =e\Bigl( \frac{hg^n}{p}\Bigr),
\end{equation}
where $h\in \Fpt$ and $g\in \Fpt$ is a primitive root modulo $p$. The
bound \eqref{eq-good} is now replaced by
\begin{equation}\label{-1}
  \corr{\varphi}{a} =
  \sum_{m=0}^{p-2} e \Bigl(  \frac{h(g^m -g^{m+a})}{p}\Bigr)= -1,
\end{equation}
when $(p-1) \nmid a$. We then get:

\begin{theorem}\label{korobovthm}  
  Let $p\geq 3$ be a prime, $h\in\Fpt$ and let $\varphi$ be the
  function defined on $\Zz/ (p-1) \Zz$ by \eqref{korobov}. Then, for
  every $\eps>0$ and every interval $I$ in $\Zz/(p-1)\Zz$, we have
\begin{equation}\label{in2}
  \bigl\vert S (\varphi; I) 
  \bigr\vert \leq (3^\frac{1}{3}+\eps) \vert I \vert 
  \Bigl(\frac{\sqrt{p}}{\vert I    \vert} 
  \Bigr)^\frac{2}{3}
\end{equation}
if $|I|$ is large enough in terms of $\eps$.
\end{theorem}

\begin{proof}
The sum  $\Sigma$ is now given by
$$
\Sigma = \sum_{a=0}^{p-2}\, \Bigl\vert\, \sum_{n\in I} e \Bigl(
\frac{ hg^{n+a}}{p} \Bigr) \Bigr\vert^2.
$$
\par
Expanding the square and appealing to \eqref{-1}, we  obtain 
$$
\Sigma = (p-1) \vert I \vert -\bigl( \vert I \vert^2 -\vert I \vert
\bigr)= p\vert I \vert -\vert I \vert^2\leq p|I|,
$$
which, as in Remark~\ref{rm-kloos}, allows us to finish the proof.
% \par
% This inequality has to be compared with \eqref{Sigma=}, and the end of
% the proof of Theorem \ref{korobovthm} is then similar to the proof of
% \eqref{in1}.
\end{proof}

\end{document}